\documentclass[a4paper]{article}
\pdfoutput=1

\usepackage{mathptmx}       
\usepackage{helvet}         
\usepackage{courier}        
\usepackage{type1cm}        
\usepackage{graphicx}       

\usepackage{array,colortbl}
\usepackage[margin=3.6cm]{geometry}
\usepackage{amsmath,amsfonts,amssymb,bm,amsthm}

\DeclareFontFamily{U}{mathx}{\hyphenchar\font45}
\DeclareFontShape{U}{mathx}{m}{n}{
      <5> <6> <7> <8> <9> <10>
      <10.95> <12> <14.4> <17.28> <20.74> <24.88>
      mathx10
      }{}
\DeclareSymbolFont{mathx}{U}{mathx}{m}{n}
\DeclareFontSubstitution{U}{mathx}{m}{n}
\DeclareMathAccent{\widecheck}{0}{mathx}{"71}

\usepackage{mathtools}
\usepackage{todonotes} 
\usepackage{booktabs,multirow}
\definecolor{darkred}{RGB}{139,0,0}
\definecolor{darkgreen}{RGB}{0,100,0}
\definecolor{darkmagenta}{RGB}{180,0,180}
\definecolor{darkblue}{RGB}{0,0,190}
\usepackage[normalem]{ulem} 

\newtheorem{theorem}{Theorem}
\newtheorem{proposition}{Proposition}
\newtheorem{lemma}{Lemma}
\newtheorem{remark}{Remark}

\usepackage{mathptmx}       
\usepackage{helvet}         
\usepackage{courier}        

\usepackage{makeidx}         
\usepackage{graphicx}        
\usepackage[bottom]{footmisc}

\usepackage{latexsym}
\usepackage{amsmath}
\usepackage{amsfonts}
\usepackage{amssymb}
\usepackage{bm}

\usepackage{url}
\usepackage{algorithm}
\usepackage{algorithmic}
\usepackage[misc,geometry]{ifsym}

\newcommand{\bsa}{{\boldsymbol{a}}}

\newcommand{\bsh}{{\boldsymbol{h}}}

\newcommand{\bsq}{{\boldsymbol{q}}}

\newcommand{\bsw}{{\boldsymbol{w}}}
\newcommand{\bsx}{{\boldsymbol{x}}}
\newcommand{\bsy}{{\boldsymbol{y}}}
\newcommand{\bsz}{{\boldsymbol{z}}}

\newcommand{\bszero}{{\boldsymbol{0}}} 

\newcommand{\bsbeta}{{\boldsymbol{\beta}}}
\newcommand{\bsgamma}{{\boldsymbol{\gamma}}}

\newcommand{\rd}{{\mathrm{d}}}

\newcommand{\bbZ}{{\mathbb{Z}}}
\newcommand{\R}{{\mathbb{R}}} 
\newcommand{\Z}{{\mathbb{Z}}} 
\newcommand{\ZZ}{{\mathbb{Z}}} 

\DeclareSymbolFont{bbold}{U}{bbold}{m}{n}
\DeclareSymbolFontAlphabet{\mathbbold}{bbold}

\newcommand{\setu}{{\mathfrak{u}}}

\DeclareSymbolFont{bbold}{U}{bbold}{m}{n}
\DeclareSymbolFontAlphabet{\mathbbold}{bbold}

\usepackage{microtype}

\newcommand{\footremember}[2]{%
	\footnote{#2}
	\newcounter{#1}
	\setcounter{#1}{\value{footnote}}%
}
\newcommand{\footrecall}[1]{%
	\footnotemark[\value{#1}]%
}

\usepackage[colorlinks=true,linkcolor=black,citecolor=black,urlcolor=black]{hyperref}
\usepackage{bookmark}
\makeatletter 
  \providecommand*{\toclevel@author}{999}
  \providecommand*{\toclevel@title}{0}
\makeatother

\begin{document}

\title{Successive Coordinate Search and Component-by-Component Construction of Rank-1 Lattice Rules}
\author{
	Adrian Ebert \footremember{idAE} {Department of Computer Science, K.U. Leuven, Celestijnenlaan 200A, B-3001 Heverlee,  Belgium.}
	\footnote{Email: Adrian.Ebert@cs.kuleuven.be}
	\and Hernan Le\"ovey \footremember{idHL} {Structured Energy Management Team, Axpo, Baden, Switzerland, Email: hernaneugenio.leoevey@axpo.com}
	\and Dirk Nuyens \footrecall{idAE} 
	\footnote{Email: Dirk.Nuyens@cs.kuleuven.be}
}
\maketitle

	\abstract{The (fast) component-by-component (CBC) algorithm is an efficient tool for the construction of generating vectors 
		for quasi-Monte Carlo rank-1 lattice rules in weighted reproducing kernel Hilbert spaces.
		We consider product weights, which assign a weight to each dimension. 
		These weights encode the effect a certain variable (or a group of variables by the product of the individual weights) has.
		Smaller weights indicate less importance. 
		Kuo \cite{Kuo_CBC:2003} proved that CBC constructions achieve the optimal rate of convergence in the respective 
		function spaces, but this does not imply the algorithm will find the generating vector with the smallest worst-case error. 
		In fact it does not.
		We investigate a generalization of the component-by-component construction that allows for a general successive 
		coordinate search (SCS), based on an initial generating vector, and with the aim of getting closer to the smallest worst-case error. 
		The proposed method admits the same type of worst-case 
		error bounds as the CBC algorithm, independent of the choice of the initial vector. 
		Under the same summability conditions on the weights as in \cite{Kuo_CBC:2003} the error bound of the algorithm can be made independent of 
		the dimension $d$ and we achieve the same optimal order of convergence for the function spaces from \cite{Kuo_CBC:2003}.
		Moreover, a fast version 
		of our method, based on the fast CBC algorithm as in \cite{Nuyens_FCBC:2006}, is available, reducing the computational cost 
		of the algorithm to $O(d \, n \log(n))$ operations, where $n$ denotes the number of function evaluations.
		Numerical experiments seeded by a Korobov-type generating vector show that the new SCS algorithm will find better choices than the CBC algorithm and the effect is better for slowly decaying weights.
	}
	
	\section{Introduction}\label{sec:1}
	
	In this article we study the numerical approximation of integrals of the form
	\begin{align*}
	I(f) = \int_{[0,1]^d} f(\bsx) \, \rd\bsx 
	\end{align*}
	for $d$-variate functions $f$ via quasi-Monte Carlo quadrature rules. Quasi-Monte Carlo rules are equal-weight
	quadrature rules of the form
	\begin{align*}
	Q_{n,d}(f)= \frac{1}{n} \sum_{k=0}^{n-1} f(\bsx_{k}) \, ,
	\end{align*}
	where the quadrature points $\bsx_0,\ldots,\bsx_{n-1} \in [0,1]^d$ are chosen deterministically. Here, we consider
	integrands $f: [0,1]^d \to \R$ which belong to some normed function space $(H,\|\cdot\|_H)$. In order to assess the quality
	of a particular QMC rule $Q_{n,d}$ with underlying point set $P_n = \{\bsx_0,\ldots,\bsx_{n-1}\}$, we introduce the notion
	of the so-called worst-case error, see, e.g., \cite{Acta_Numerica_QMC:2013}, defined by
	\begin{align*}
	e_{n,d}(P_n,H) = \sup_{\|f\|_{\stackrel{}{H}} \leq 1} \left| \int_{[0,1]^d} f(\bsx) \, \rd\bsx - \frac{1}{n} \sum_{{k}=0}^{n-1} f(\bsx_{k}) \right| \, .
	\end{align*}
	In other words, $e_{n,d}(P_n,H)$ is the worst error that is attained over all functions in the unit ball of $H$ using 
	the quasi-Monte Carlo rule with quadrature points in $P_n$.
	It is often possible to obtain explicit expressions to calculate $e_{n,d}(P_n,H)$, see, e.g., \cite{Nuyens:2014}.
	In particular, we consider weighted Korobov and weighted shift-averaged Sobolev spaces, 
	which are both reproducing kernel Hilbert spaces, for details see, e.g., \cite{Sloan_and_Woz_Sobolev:1998,Sloan_and_Woz_Korobov:2001,Kuo_CBC:2003,Acta_Numerica_QMC:2013}.
	In this paper we will limit ourselves to the original choice of ``product weights''.
	In essence, the idea is to quantify the varying importance of the coordinate directions $x_j$ with $j=1,\ldots,d$ w.r.t.\ the 
	function values by a sequence $\bsgamma=\{\gamma_j\}_{j=1}^d$ of positive weights. 
	
	There are many ways to choose the underlying point set $P_n$ of a QMC rule, ranging from lattice rules and sequences, digital nets and sequences and more recent constructions such as interlaced polynomial lattice rules. In this paper, however, we will restrict ourselves to rank-1 lattice rules.
	This type of QMC rules has an underlying point set $P_n \subseteq [0,1]^d$ of the form
	\begin{align*}
	P_n 
	&= 
	\left\{ \left\{ \frac{k \, \bsz}{n} \right\} \, \middle| \,\, 0 \leq k < n \,\, \right\} \, ,
	\end{align*}
	where $\bsz \in \Z^d$ is the \textit{generating vector} of the rank-1 lattice rule and $\{\cdot\}$ denotes the fractional part, componentwise if applied to a vector. It is clear that any vector congruent modulo~$n$ is equivalent and so we only consider values modulo~$n$.
	The components of  $\bsz$ are often restricted to the set of integers in $\{1,\ldots,n-1\}$ that are relatively prime to $n$,
	see, e.g., \cite{Kuo_CBC:2003}, such that one obtains $n$ distinct points for all one-dimensional projections, and as such for any projection. 
	Here we consider generating vectors $\bsz \in \bbZ_n^d$ with n prime and
	\begin{align*}
	\bbZ_n = \left\{0, 1,2,\ldots,n-1 \right\} .
	\end{align*}
	For rank-1 lattice rules in a weighted shift-invariant tensor-product reproducing kernel Hilbert space $H(K)$ with 
	reproducing kernel 
	$K(\bsx,\bsy) = \prod_{j=1}^d (\beta_j + \gamma_j \, \omega(x_j-y_j))$
	the squared worst-case error can be written as
	\begin{align*}
	e_{n,d}^2(P_n,H)
	=
	e_{n,d}^2(\bsz) 
	&= 
	-\prod_{j=1}^d \beta_j + \frac{1}{n} \sum_{k=0}^{n-1} \prod_{j=1}^d \left( \beta_j + \gamma_j \, 
	\omega\left( \left\{ \frac{k \, z_j}{n} \right\} \right) \right) \, ,
	\end{align*}
	with positive weights $\bsgamma = \{\gamma_j\}_{j=1}^d$ and $\bsbeta = \{\beta_j\}_{j=1}^d$, and where $\int_0^1 \omega(t) \, \rd t = 0$, see, e.g., \cite{Nuyens_FCBC:2006}.
	We note that the weights $\beta_j$ are to easily accommodate for some types of shift-averaged Sobolev spaces.
	Moreover, the initial squared worst-case error in this function space, i.e., using $n=0$ samples and with the convention that $Q_{0,d}(f) = 0$, is given by
	\begin{align*}
	e_{0,d}^2(0,H) = e_{0,d}^2 = \prod_{j=1}^d \beta_j .
	\end{align*}
	
	\begin{remark}\label{remark:normalized}
		It is always possible to consider the normalized worst-case error by dividing by the initial worst-case error for the zero-algorithm. 
		The squared normalized worst-case error then takes the form
		\begin{align}\label{eq:normalized}
		\frac{e^2_{n,d}(\bsz)}{e^2_{0,d}}
		&=
		-1 + \frac{1}{n} \sum_{k=0}^{n-1} \prod_{j=1}^d \left( 1 + \frac{\gamma_j}{\beta_j} \, \omega\left( \left\{ \frac{k \, z_j}{n} \right\} \right) \right)
		\, ,
		\end{align}
		with $e^2_{0,d} = \prod_{j=1}^d \beta_j$. This is equivalent to considering the worst-case error $e_{n,d}(\bsz)$ with modified weight 
		sequences $\hat\beta_j = 1$ and $\hat{\gamma}_j = \gamma_j/\beta_j$. 
	\end{remark}
	
	One of the most commonly considered methods to construct good rank-1 lattice rules is the component-by-component (CBC) construction, see, e.g., \cite{Nuyens:2014}, which extends the generating vector one component at a time by selecting the next components $z_s$ which minimizes the worst-case error of the $s$-dimensional rule. The pseudo-code of the CBC algorithm is given below.
	
	{\centering
	\begin{minipage}{\linewidth}
	\begin{algorithm}[H]
		\caption{Component-by-component construction (CBC)}
		\label{alg:FormCBCAlg}
		\begin{algorithmic}
			\STATE \bf{Output:} $\bsz \in \bbZ_n^d$ \\[2mm]
			\FOR{$s=1$ \bf{to} $d$}
			\FORALL{$z_s \in \bbZ_n$}
			\STATE $\displaystyle e^2_{n,s}(z_1,\ldots,z_{s-1},z_s) = -\prod\limits_{j=1}^{s} \beta_j + \frac{1}{n} \sum\limits_{k=0}^{n-1}
			\prod\limits_{j=1}^{s} \left(\beta_j + \gamma_j \, \omega\left(\left\{ \frac{k\, z_j}{n}
			\right\}\right)\right)$
			\ENDFOR
			\STATE $z_{s} = \underset{z \in \bbZ_n}{\operatorname*{argmin}} \,\, e^2_{n,s}(z_1,\ldots,z_{s-1},z)$
			\ENDFOR
		\end{algorithmic}
	\end{algorithm}
	\end{minipage}
	\par
	}
	\vspace{0.5cm}
	It was shown in \cite{Kuo_CBC:2003} that the component-by-component construction generates lattice rules
	which achieve optimal rates of convergence in weighted Korobov and Sobolev function spaces. Additionally, a
	fast construction method is available, see \cite{Nuyens_FCBC:2006,Nuyens_FCBC_nonprime:2006}, that reduces the construction cost to
	$O(d \, n \, \log(n))$ operations.
	
	Even though the CBC algorithm constructs generating vectors $\bsz$ which exhibit the optimal error asymptotics,
	the constructed  vector is not necessarily the one minimizing the worst-case error $e_{n,d}(\bsz)$.
	We will therefore introduce and investigate a different construction method which can generate lattice rules with a 
	smaller worst-case error than the CBC construction.
	
	The article is structured as follows. In Section~\ref{sec:SCS} we introduce the successive coordinate search (SCS) algorithm and
	analyse some properties. In Section~\ref{sec:bounds} we prove that the SCS construction achieves optimal rates of convergence in
	the weighted Korobov and weighted shift-averaged Sobolev space. 
	To get dimension-independent bounds, i.e., achieve tractability, we show that the summability condition on the weights is the same as for the normal CBC construction. Finally we report on various numerical experiments in Section~\ref{sec:numerics}.    
	
	
	\section{Formulation of the successive coordinate search algorithm}\label{sec:SCS}
	
	In this section we introduce an algorithm of similar nature to the component-by-component construction. One advantage of the component-by-component construction is that the algorithm is extensible in the dimension~$d$, i.e., 
	to find the $(d+1)$-dimensional generating vector, the algorithm does not need to restart but just starts from the generating vector of dimension~$d$.
	In our setting this also implies that a $d$-dimensional vector, with $d$ large enough, could be constructed and used for all problems with less than $d$~dimensions.
	This allows us to fix the maximum number of dimensions to some large enough~$d$ and successively try to find the best $s$-th component of a $d$-dimensional generating vector, keeping all other $d-1$ choices fixed. The pseudocode of the successive coordinate search (SCS) algorithm is given below.
	
	{\centering
	\begin{minipage}{\linewidth}
	\begin{algorithm}[H]
		\caption{Successive coordinate search algorithm (SCS)}
		\label{alg:FormSCSAlg}
		\begin{algorithmic}
			\STATE \bf{Input:} $\bsz^0 \in \bbZ_n^d$
			\STATE \bf{Output:} $\bsz \in \bbZ_n^d$ \\[2mm]
			\FOR{$s=1$ \bf{to} $d$}
			\FORALL{$z_{s} \in \bbZ_n$}
			\STATE $e^2_{n,d}(z_1,\ldots,z_{s-1},z_{s},z_{s+1}=z_{s+1}^0,\ldots,z_d=z_d^0)$ \quad with 
			\STATE $e^2_{n,d}(z_1,\ldots,z_d) = -\prod\limits_{j=1}^d \beta_j + \frac{1}{n} \sum\limits_{k=0}^{n-1} \prod\limits_{j=1}^d 
			\left(\beta_j + \gamma_j \, \omega\left(\left\{ \frac{k\, z_j}{n} \right\}\right)\right)$
			\ENDFOR
			\STATE $z_{s} = \underset{z \in \bbZ_n}{\operatorname*{argmin}} \,\, e^2_{n,d}(z_1,\ldots,z_{s-1},z,z_{s+1}^0,\ldots,z_d^0)$
			\ENDFOR
		\end{algorithmic}
	\end{algorithm}
	\end{minipage}
	\par
	}
	\vspace{0.5cm}
	Instead of increasing the dimension in each step of the algorithm, we keep $d$ fixed during all calculations. Based on 
	a starting vector $\bsz^0 \in \bbZ_n^d = \{0,1,\ldots,n-1\}^d$, the algorithm successively selects the coordinate $z_{s} \in \bbZ_n$ which minimizes the squared worst-case error $e^2_{n,d}(z_1,\ldots,z_{s-1},z_{s},z_{s+1}^0,\ldots,z_d^0)$ while keeping all other coordinates of $\bsz$ fixed.
	Thus, in the process of the SCS algorithm the coordinates of the starting vector $\bsz^0$ are altered in each step of the algorithm.
	Our construction is very similar to the component-by-component construction, with the only difference being that an initial vector $\bsz^0$ is required as input for the algorithm. In fact, we can prove that the successive coordinate search algorithm
	is a generalized version of the CBC algorithm as the following theorem shows by starting from an initial vector $\bsz^0 = (0,\ldots,0) \in \bbZ_n^d$. We note that this is a degenerate vector as it generates only a $1$-point rule, and thus is in some sense the worst possible choice for any~$n \ge 1$.
	
	\begin{theorem}
		The component-by-component (CBC) algorithm and the successive coordinate search (SCS) algorithm with starting vector $\bsz^0 = (0,\ldots,0)$
		both yield the same generating vector as outcome (with equivalent choices selected in the same way in both algorithms).
	\end{theorem}
	
	\begin{proof} Denote by $0^r$ the $r$-dimensional zero vector, where $1 \leq r \leq d$. For an arbitrary $\bsz \in \bbZ_n^{s}$ with $1 \leq s \leq d$ and with $\tilde{\bsz}=(\bsz,0^{d-s}) \in \bbZ_n^{d}$, the squared worst-case error equals
		\begin{align*}
		e_{n,d}^2(\tilde{\bsz}) 
		&=
		e_{n,d}^2(\bsz,0^{d-s}) = -\prod\limits_{j=1}^d \beta_j + \frac{1}{n}\sum\limits_{k=0}^{n-1} \prod\limits_{j=1}^d 
		\left(\beta_j + \gamma_j \, \omega\left(\left\{ \frac{k \, \tilde{z}_j}{n} \right\}\right)\right) \\
		&= 
		-\prod\limits_{j=1}^d \beta_j + \frac{1}{n} \sum\limits_{k=0}^{n-1} \prod\limits_{j=1}^{s} \left(\beta_j +\gamma_j \, \omega\left(\left\{ 
		\frac{k\, z_j}{n} \right\}\right)\right) \prod\limits_{j=s+1}^{d} \left(\beta_j + \gamma_j \, \omega(0)\right) \\
		&=
		-\prod\limits_{j=1}^d \beta_j+ \frac{C_s}{n} \sum\limits_{k=0}^{n-1} \prod\limits_{j=1}^{s} \left(\beta_j + 
		\gamma_j \, \omega\left(\left\{ \frac{k\, z_j}{n} \right\}\right)\right) \\  
		&=
		-\prod\limits_{j=1}^d \beta_j + C_s \left(e_{n,s}^2(\bsz) + \prod\limits_{j=1}^s \beta_j \right) \, ,
		\end{align*}
		where $C_s = \prod_{j=s+1}^{d} \, ( \beta_j + \gamma_j \, \omega(0) )$. Note that due to the non-negativity of the 
		squared worst-case error $e_{n,d}^2$ the function $\omega$ is such that $\omega(0) \ge 0$ and so the constants $C_s$ 
		are positive for all $s=1,\ldots,d$.
		
		Now, in each step $1 \leq s \leq d$ of the SCS algorithm with initial vector $\bsz^0=(0,\ldots,0)$, we search for the $z_s \in \bbZ_n$ that minimizes $e_{n,d}^2(z_1,\ldots,z_{s-1},z_s,0^{d-s})$, where $z_1,\ldots,z_{s-1}$ have been determined in the previous steps of the algorithm. By the above identity we have that
		\begin{align*}
		e_{n,d}^2(z_1,\ldots,z_{s-1},z_s,0^{d-s})
		&= 
		-\prod\limits_{j=1}^d \beta_j + C_s \left(e_{n,s}^2(z_1,\ldots,z_{s-1},z_s) + \prod\limits_{j=1}^s \beta_j \right) \, ,
		\end{align*}
		and so, since the remaining terms on the right-hand side are independent of $z_s$, this is equivalent to finding $z_s \in \bbZ_n$ such that $e_{n,s}^2(z_1,\ldots,z_{s-1},z_s)$ is minimized. As this is exactly the same quantity which is minimized in each step of
		the component-by-component construction algorithm, we see that the CBC algorithm and the SCS algorithm with starting vector $\bsz^0 = (0,\ldots,0)$ yield exactly the same outcome under the assumption that both algorithms select the same minimizer whenever multiple choices occur 
		in a minimization step.
	\end{proof}

	Furthermore, the formulation of the successive coordinate search construction guarantees that the generating vector $\bsz$ obtained 
	by the SCS algorithm with initial vector $\bsz^0$ is never worse than the input vector $\bsz^0$. 
	\begin{proposition} \label{prop:ImproveSCS}
		Let $\bsz^0 \in \bbZ_n^d$ be an arbitrary generating vector for a rank-1 lattice rule and denote by $\bsz^1 \in \bbZ_n^d$ the generating vector 
		constructed by the SCS algorithm with starting vector $\bsz^0 $. Then we have that
		\begin{align*}
		e_{n,d}(\bsz^1) \leq e_{n,d}(\bsz^0) \, ,
		\end{align*}
		i.e., the SCS method constructs a generating vector with worst-case error smaller than or equal to the worst-case error of the initial vector.
	\end{proposition}
	\begin{proof}
		The statement follows directly from the formulation of the algorithm.
	\end{proof}
	
	Similar to the case of the component-by-component construction there is a fast version available that allows for the construction of generating vectors with time complexity $O(d \, n \, \log(n))$. In case $n$ is a prime number this results in the following algorithm.
	
	{\centering
	\begin{minipage}{\linewidth}
	\begin{algorithm}[H] 
		\caption{Fast version of the SCS algorithm for prime $n$}
		\begin{algorithmic}
			\STATE \bf{Input:} $\bsz^0 \in \bbZ_n^d$
			\vspace{2pt}
			\STATE $\mathbf{q}_0 = \mathbf{1}_{n \times 1} \in \R^n$
			\FOR{$s=1$ \bf{to} $d$}
			\STATE $\mathbf{q}_{s} = \bigl( \beta_{s} \, \mathbf{1}_{1 \times n} + \gamma_{s} \;\Omega_n^{\langle g\rangle}\!(z_{s}^0,:) \bigr) \mathbin{.*} \mathbf{q}_{s-1} $ \hspace{\fill} $\triangleright$ {initialize $\mathbf{q}$}
			\ENDFOR
			\FOR{$s=1$ \bf{to} $d$}
			\STATE $\mathbf{q}_{d} = \mathbf{q}_{d} \, \mathbin{./} \bigl( \beta_{s} \, \mathbf{1}_{1 \times n} + \gamma_{s} \; \Omega_n^{\langle g\rangle}\!(z_{s}^0,:) \bigr)$			
			\hspace{\fill} $\triangleright$ {divide out initial choice $z_s^0$}
			\STATE$\mathbf{E}_{s}^2 = - \overline{\beta}_{s} \, \mathbf{1}_{n \times 1} + \frac1n \bigl(\beta_{s} \, \mathbf{1}_{n \times n} + \gamma_{s} \; \Omega_n^{\langle g\rangle}\bigr) \, \mathbf{q}_{d}$ 
			\hspace{\fill} $\triangleright$ {FFT for matrix-vector product}
			\STATE $z_{s} = \operatorname*{argmin}_{z \in \bbZ_n} \,\, E_{s}^2(z)$
			\hspace{\fill} $\triangleright$ {select component}
			\STATE $\mathbf{q}_{d}  = \bigl( \beta_{s} \, \mathbf{1}_{1 \times n} + \gamma_{s} \; \Omega_n^{\langle g\rangle}\!(z_{s},:) \bigr) \mathbin{.*} \mathbf{q}_{d}$
			\hspace{\fill} $\triangleright$ {update $\bsq$ with new choice $z_s$}
			\ENDFOR
		\end{algorithmic}
	\end{algorithm}
	\end{minipage}
	\par
	}
	\vspace{0.5cm}
	Here we used the notations
	\begin{align*}
	\overline{\beta}_{s} = \prod_{j=1}^{s} \beta_j, 
	\qquad 
	\Omega_n = \left[\omega\left(\left\{\frac{k \, \bsz}{n}\right\}\right)\right]_{\substack{z=0,\ldots,n-1 \\ k = 0,\dots,n-1}}
	\end{align*}
	and $\Omega_n^{\langle g\rangle}$ denotes the reordering of $\Omega_n$ w.r.t.\ a generator $g$ for the cyclic group of $\bbZ_n$. 
	For more details see \cite{Nuyens_FCBC:2006,Nuyens:2014}.
	The symbols $\mathbin{.*}$ and $\mathbin{./}$ denote componentwise vector multiplication and division, respectively, 
	and $\Omega_n^{\langle g\rangle}\!(j,:)$
	stands for the $j$-th row of $\Omega_n^{\langle g\rangle}$.
	Note that the computation is only slightly more expensive than the fast CBC algorithm since $\mathbf{q}$ has to be initialized and updated using $\bsz^0$, the computational complexity is still $O(d \, n \log(n))$.
	
	
	\section{Error Bounds for the SCS Algorithm}\label{sec:bounds}
	
	In this section we derive worst-case error bounds and show that the previously introduced successive coordinate search construction achieves optimal convergence rates in the respective function space. Here, we consider two of the most common function spaces in QMC theory, the weighted Korobov space and the weighted shift-averaged (anchored) Sobolev space. 
	
	\subsection{The weighted Korobov space}		
	
	Let $\bsgamma =\{\gamma_j\}$ and $\bsbeta = \{\beta_j\}$ be two weight sequences. 
	The reproducing kernel of the corresponding $d$-dimensional weighted Korobov space is then given by
	\begin{align*}
	K_{d,\bsgamma,\bsbeta}(\bsx, \bsy) 
	&=  
	\prod_{j=1}^d \left( \beta_j + \gamma_j \sum_{0 \ne h=-\infty}^{\infty} \frac{\mathrm{e}^{2\pi \mathrm{i} \, h (x_j - y_j)}}{r_\alpha(h)} \right)\, ,
	\end{align*}
	where 
	$\alpha > \tfrac12$ is referred to as the smoothness parameter and we define 
	\begin{align*}
	r_\alpha(\bsh) 
	= 
	\prod_{j=1}^d
	r_\alpha(h_j)
	,
	\qquad
	r_\alpha(h)
	=
	\begin{cases}
	|h|^{2\alpha}, & \text{if $h\ne0$}, \\
	1, & \text{otherwise}.
	\end{cases}
	\end{align*}
	For integer $\alpha$ the smoothness can be interpreted as the number of mixed partial derivatives $f^{(\tau_1,\ldots,\tau_d)}$ with $(\tau_1,\ldots,\tau_d) \le (\alpha,\ldots,\alpha)$ that exist and are square-integrable. 
	The space consists of functions which can be represented as absolutely summable Fourier series with norm
	\begin{align*}
	\|f\|_{K_{d,\bsgamma,\bsbeta}}^2
	=
	\sum_{\bsh \in \ZZ^d} |\hat{f}_\bsh|^2 \, r_\alpha(\bsh) \,
	\prod_{\substack{j=1 \\ h_j \ne 0}}^d \gamma_j  
	\prod_{\substack{j=1 \\ h_j = 0}}^d \beta_j
	,
	\end{align*}
	where the $\hat{f}_\bsh$ denote the Fourier coefficients of $f$.
	
	We prove that the successive coordinate search (SCS) algorithm achieves the optimal rate of convergence for multivariate 
	integration in the weighted Korobov space for any initial vector.
	As is usual practice, we restrict ourselves to a prime number of points to simplify the needed proof techniques.
	We need the following lemma in the proof of the theorem.
	
	\begin{lemma}\label{lem:averaging}
		For $s \in \{1,\ldots,d\}$, $n$ prime, arbitrary integers $z_j$, $j \in \{1\mathpunct{:}d\} \setminus \{s\}$, and $r(\bsh) = \prod_{j=1}^d r(h_j)$ with $r(h) > 0$ such that for $h\ne0$ we have $r(n h) \ge n^c \, r(h)$ for $c \ge 1$, then
		\begin{align*}
		0
		\le
		\frac1n \sum_{z_s=0}^{n-1} 
		\sum_{\substack{\bsh \in (\ZZ \setminus \{0\})^d \\ \bsh \cdot \bsz \equiv 0~(\operatorname{mod} n)}} r^{-1}(\bsh) 
		\le
		\frac2n \sum_{\bsh \in (\ZZ \setminus \{0\})^d} r^{-1}(\bsh)
		.
		\end{align*}
	\end{lemma}
	\begin{proof}
		The condition $\bsh \cdot \bsz \equiv 0 \pmod{n}$ can be written equivalently by
		\begin{align*}
		\frac1n \sum_{z_s=0}^{n-1} 
		\sum_{\substack{\bsh \in (\ZZ \setminus \{0\})^d \\ \bsh \cdot \bsz \equiv 0~(\operatorname{mod} n)}} r^{-1}(\bsh) 
		=
		\frac1n \sum_{z_s=0}^{n-1} \,\,
		\sum_{\bsh \in (\ZZ \setminus \{0\})^d } r^{-1}(\bsh)
		\, \left[\frac1n \sum_{k=0}^{n-1} \prod_{j=1}^d \mathrm{e}^{2\pi \mathrm{i} \, k \, h_j z_j / n}\right]
		.
		\end{align*}
		Further, for $0\le k < n$ and $n$ prime
		\begin{align*}
		\frac1n \sum_{z=0}^{n-1} \mathrm{e}^{2\pi \mathrm{i} \, k h z / n}
		=
		\begin{cases}
		1, & \text{if $k=0$ or } h \equiv 0 \pmod{n} , \\
		0, & \text{otherwise},
		\end{cases}
		\end{align*}
		and hence for $h \in \Z$
		\begin{align*}
		\frac1n \sum_{k=0}^{n-1} \frac1n \sum_{z=0}^{n-1} \mathrm{e}^{2\pi \mathrm{i} \, k h z / n}
		=
		\begin{cases}
		1, & \text{if } h \equiv 0 \pmod{n} ,\\
		1/n, & \text{if } h \not\equiv 0 \pmod{n} .
		\end{cases}
		\end{align*}
		Thus
		\begin{align*}
		& 
		\sum_{\bsh \in (\ZZ \setminus \{0\})^d} r^{-1}(\bsh) \, \frac1n \sum_{k=0}^{n-1} \left( \prod_{s\ne j=1}^d \mathrm{e}^{2\pi \mathrm{i} \, k \, h_j z_j / n} \right) \frac1n \sum_{z_s=0}^{n-1} \mathrm{e}^{2\pi \mathrm{i} \, k \, h_s z_s / n} 
		\\
		&\quad \le
		\sum_{\bsh \in (\ZZ \setminus \{0\})^d} r^{-1}(\bsh) \, \frac1n \sum_{k=0}^{n-1} \left( \prod_{s\ne j=1}^d \left| \mathrm{e}^{2\pi \mathrm{i} \, k \, h_j z_j / n} \right| \right) \left| \frac1n \sum_{z_s=0}^{n-1} \mathrm{e}^{2\pi \mathrm{i} \, k \, h_s z_s / n} \right|
		\\
		&\quad \le
		\sum_{\bsh \in (\ZZ \setminus \{0\})^d} r^{-1}(\bsh) \, \frac1n \sum_{k=0}^{n-1}  \frac1n \sum_{z_s=0}^{n-1} \mathrm{e}^{2\pi \mathrm{i} \, k \, h_s z_s / n} 
		\\
		&\quad =
		\sum_{\substack{\bsh \in (\ZZ \setminus \{0\})^d \\ h_s \equiv 0 \pmod{n}}} r^{-1}(\bsh)  
		+
		\frac1n \sum_{\substack{\bsh \in (\ZZ \setminus \{0\})^d \\ h_s \not\equiv 0 \pmod{n}}} r^{-1}(\bsh)  
		\\
		&\quad \le
		\frac1n \sum_{\bsh \in (\ZZ \setminus \{0\})^d} r^{-1}(\bsh)  
		+
		\frac1n \sum_{\substack{\bsh \in (\ZZ \setminus \{0\})^d \\ h_s \not\equiv 0 \pmod{n}}} r^{-1}(\bsh)  
		\le
		\frac2n \sum_{\bsh \in (\ZZ \setminus \{0\})^d} r^{-1}(\bsh)
		\,.
		\end{align*}
		Which completes the proof.
	\end{proof}
	
	\begin{theorem} \label{thm:ErrBouKor}
		Let $n$ be a prime number and $\bsz^0 = (z_1^0,\ldots,z_d^0) \in \bbZ_n^d$ be an arbitrary initial vector.
		Furthermore, denote by $\bsz^{*} = (z_1^{*},\ldots,z_d^{*})  \in \bbZ_n^d$ the generating vector constructed by the successive
		coordinate search method with initial vector $\bsz^0$. Then the squared worst-case error $e_{n,d}^2(\bsz^*)$ in the Korobov space with kernel $K_{d,\bsgamma,\bsbeta}$ satisfies
		\begin{align*}
		e_{n,d}^2(\bsz^{*})
		&\leq 
		C_{d,\lambda} \, n^{- \lambda} \quad \text{for all} \quad 1 \leq \lambda < 2\alpha 
		\, ,
		\end{align*}
		where the constant $C_{d,\lambda}$ is given by
		\begin{align*}
		C_{d,\lambda}
		&=
		2^\lambda
		\left(
		\sum_{j=1}^d 
		\frac{\gamma_j^{1/\lambda}}{\beta_j^{1/\lambda}}
		\right)^\lambda
		\prod_{j = 1}^d
		\left(
		\beta_j^{1/\lambda}+
		\gamma_j^{1/\lambda}
		\mu_{\alpha,\lambda}
		\right)^\lambda
		\mu_{\alpha,\lambda}^\lambda
		\max_{s=1,\ldots,d} \left(
		1+
		\frac{\gamma_s^{1/\lambda}}{\beta_s^{1/\lambda}}
		\mu_{\alpha,\lambda}
		\right)^{-\lambda}
		\end{align*}
		with $\mu_{\alpha,\lambda} = 2\zeta(2\alpha/\lambda) \,. $
		Additionally, if the weights satisfy the summability condition
		\begin{align*}
		\sum_{j=1}^{\infty} \frac{\gamma_j^{\,1/\lambda}}{\beta_j^{\,1/\lambda}} < \infty
		\end{align*}
		then $C_{d,\lambda} \le C_\lambda < \infty$, and the constant $C_\lambda$ is bounded independent of the dimension~$d$. 
		Hence, the worst-case error $e_{n,d}(\bsz^{*})$ can be taken arbitrarily close to $O(n^{-\alpha})$, with the implied constant independent of $n$, and independent of $d$ if the summability condition holds.
	\end{theorem}
	
	\begin{proof}
		We use the notation from \cite{Nuyens:2014}:
		for a subset $\setu \subseteq \{1\mathpunct{:}d\} = \{1,\ldots,d\}$, 
		we define the set $\ZZ_\setu = \{ \bsh \in \ZZ^d : h_j \ne 0 \text{ for all } j \in \setu \text{ and } h_j = 0 \text{ for } j \not\in \setu \}$, 
		and define the dual lattice 
		$L^\perp_\setu(\bsz,n) = L^\perp_\setu(\bsz_\setu, n) = \{ \bsh \in \ZZ_\setu : \bsh_\setu \cdot \bsz_\setu \equiv 0 \pmod{n} \}$ where we write $\bsz_\setu$ and $\bsh_\setu$ to refer only to those components in $\bsz$ and $\bsh$.
		For $\bsh \in \bbZ_\setu$ we will write $\bsh_\setu \in \bbZ_\setu$ and $r_\alpha(\bsh_\setu)$ to explicitly denote the dependence on the dimensions in $\setu$ only. We also write $\gamma_u = \prod_{j \in u} \gamma_j$ and set $\gamma_\emptyset = 1$.
		Now, without loss of generality, we consider the case where $\beta_j=1$ for all $j$ and correct the final expression afterwards, see Remark~\ref{remark:normalized}. Then from~\eqref{eq:normalized}, or, see, e.g., \cite[p.~5, eq.~(6)]{Nuyens:2014} with $q=2$ and $\varphi(\bsx_k) = \mathrm{e}^{2\pi \mathrm{i} \, k \, \bsh \cdot \bsz/n }$, we have
		\begin{align*}
		e_{n,d}^2(\bsz)
		=
		\sum_{\bszero \neq \bsh \in \Z^d} 
		\left| \frac{1}{n} \sum_{k=0}^{n-1} \mathrm{e}^{2\pi \mathrm{i} \, k \, \bsh \cdot \bsz/n }\right|^2 
		r_{\alpha}^{-1}(\bsh) \prod_{\substack{j=1\\h_j\ne0}}^d \gamma_j 
		&=
		\sum_{\emptyset \ne \setu \subseteq \{1\mathpunct{:}d\}}
		\gamma_\setu
		\sum_{\bsh_\setu \in L^\perp_\setu(\bsz_\setu,n)} r_\alpha^{-1}(\bsh_\setu) \, .
		\end{align*}
		Now define
		\begin{align*}
		g_\setu(\bsz_\setu) 
		&=
		\gamma_\setu
		\sum_{\bsh_\setu \in L^\perp_\setu(\bsz_\setu,n)} r_\alpha^{-1}(\bsh_\setu) 
		\qquad \text{and} \qquad
		T_s(z_1,\ldots,z_s) = \sum_{s \in \setu \subseteq \{1\mathpunct{:}s\}} g_\setu(\bsz_\setu)  
		\,,
		\end{align*}
		which gives
		\begin{align*}
		e_{n,d}^2(\bsz)
		&=
		\sum_{\emptyset \ne \setu \subseteq \{1\mathpunct{:}d\}} g_\setu(\bsz_\setu)
		= \sum_{s=1}^{d} \sum_{s \in \setu \subseteq \{1\mathpunct{:}s\}} g_\setu(\bsz_\setu) 
		= \sum_{s=1}^d T_s(z_1,\ldots,z_s)
		\, .
		\end{align*}
		Minimizing $e_{n,d}(\bsz)$ over $z_s \in \bbZ_n$
		is equivalent to minimizing only those parts which depend on $z_s$, resulting in the auxiliary target function
		\begin{align*}
		\theta_{s}(\bsz)
		&=
		\sum_{s \in \setu \subseteq \{1\mathpunct{:}d\}} g_\setu(\bsz_\setu)
		=
		\sum_{s \in \setu \subseteq \{1\mathpunct{:}d\}} \gamma_\setu
		\sum_{\bsh_\setu \in L^\perp_\setu(\bsz_\setu,n)} r_\alpha^{-1}(\bsh_\setu) 
		\,.
		\end{align*}
		We note that in the standard CBC proofs this quantity only depends on the dimensions up to $s$ while here it depends on all $d$~dimensions.
		Obviously
		\begin{align*}
		e_{n,d}^2(\bsz)
		&= 
		\sum_{s=1}^d T_s(z_1,\ldots,z_s)
		=
		\sum_{s=1}^{d} \sum_{s \in \setu \subseteq \{1\mathpunct{:}s\}} g_\setu(\bsz_\setu) 
		\le
		\sum_{s=1}^{d} \sum_{s \in \setu \subseteq \{1\mathpunct{:}d\}} g_\setu(\tilde{\bsz}_\setu) 
		=
		\sum_{s=1}^{d} \theta_s(\tilde{\bsz}) 
		\, ,
		\end{align*}
		where the tilde on top of $\bsz$ means that in replacing the sum over $s \in \setu \subseteq \{1\mathpunct{:}s\}$ by the sum over $s \in \setu \subseteq \{1\mathpunct{:}d\}$ we choose arbitrary $z_j$ for $j > s$. We are free to do so since we are just adding positive quantities. Furthermore, for $1 \le \lambda < \infty$, using the so-called Jensen's inequality, we obtain
		{\allowdisplaybreaks
			\begin{align*}
			\left(e_{n,d}^2(\bsz)\right)^{1/\lambda}
			=
			\left(\sum_{s=1}^{d} \sum_{s \in \setu \subseteq \{1\mathpunct{:}s\}} g_\setu(\bsz_\setu)\right)^{1/\lambda}
			&\le
			\left(\sum_{s=1}^{d} \sum_{s \in \setu \subseteq \{1\mathpunct{:}d\}} g_\setu(\tilde{\bsz}_\setu)\right)^{1/\lambda}
			\\
			&=
			\left(\sum_{s=1}^{d}  \theta_s(z_1,\ldots,z_{s-1},z_s,w_{s+1},\ldots,w_d) \right)^{1/\lambda}
			\\
			&\le
			\sum_{s=1}^{d}  \theta^{1/\lambda}_s(z_1,\ldots,z_{s-1},z_s,w_{s+1},\ldots,w_d)
			\, ,
			\end{align*}}
		which holds for all choices of $\bsw$.
		Since in minimizing $e^2_{n,d}(z_1^*,\ldots,z_{s-1}^*,z_s,z_{s+1}^0,\ldots,z_d^0)$ we are in fact minimizing $\theta_{s}(z_1^{*},\ldots,z_{s-1}^{*},z_s,z_{s+1}^{0},\ldots,z_{d}^{0})$ we now use the standard reasoning 
		that the best choice $z_s = z_s^*$ makes $\theta_s$ at least as small as the average over all choices, and 
		the same reasoning holds if we raise $\theta_s$ to the power $1/\lambda$. Therefore we obtain
		\begin{align*}
		&\theta_{s}^{1/\lambda}(z_1^{*},\ldots,z_{s-1}^{*},z_s^*,z_{s+1}^{0},\ldots,z_{d}^{0})
		\le
		\frac1{n} \sum_{z_s \in \ZZ_n} \theta_{s}^{1/\lambda}(z_1^{*},\ldots,z_{s-1}^{*},z_s,z_{s+1}^{0},\ldots,z_{d}^{0})
		\\
		&\quad \leq
		\sum_{s \in \setu \subseteq \{1\mathpunct{:}d\}} 
		\gamma_\setu^{1/\lambda}
		\frac1n \sum_{z_s \in \ZZ_n} 
		\sum_{\bsh_\setu \in L^\perp_\setu(\bar{\bsz}_\setu,n)} r_\alpha^{-1/\lambda}(\bsh_\setu) 
		\leq
		\frac2n \sum_{s \in \setu \subseteq \{1\mathpunct{:}d\}} 
		\gamma_\setu^{1/\lambda}
		\sum_{\bsh_\setu \in \Z_{\setu}} r_\alpha^{-1/\lambda}(\bsh_\setu)
		\, ,
		\end{align*}
		where we used Jensen's inequality to obtain the second line (and where $\bar{\bsz}_\setu$ means we take $\bar\bsz = (z_1^*$,\ldots,$z_{s-1}^*,z_s,z_{s+1}^0,\ldots,z_{s+1}^d)$) and Lemma~\ref{lem:averaging}, relabeling the set $\setu$ to be $\{1,\ldots,|\setu|\}$, $d=|\setu|$, and with $r(h) = |h|^{2\alpha/\lambda}$ and $c=2\alpha/\lambda \ge 1$, to obtain the last line.
		For convenience we define
		\begin{align*}
		\mu_{\alpha,\lambda}
		=
		\sum_{0 \ne h \in \Z} r_\alpha^{-1/\lambda}(h)
		=
		2 \sum_{h=1}^\infty h^{-2\alpha/\lambda}
		=
		2 \zeta(2\alpha/\lambda)
		<
		\infty
		\,
		\end{align*}
		from which it follows that $2\alpha/\lambda > 1$ and we thus need $\lambda < 2\alpha$.
		Since for $\bsh_\setu \in \Z_\setu$ we have $r_\alpha^{-1/\lambda}(\bsh_\setu) = \prod_{j\in\setu} r_\alpha^{-1/\lambda}(h_j)$
		we find $\sum_{\bsh_\setu \in \Z_{\setu}} r_\alpha^{-1/\lambda}(\bsh_\setu) = \mu_{\alpha,\lambda}^{|\setu|}$.
		
		In each step of the SCS algorithm we now have a bound on $\theta_s^{1/\lambda}$ which we insert in our bound for the worst-case error, each time choosing the components of $\bsz^0$ for $\bsw$, to obtain
		\begin{align*}
		\left(e_{n,d}^2(\bsz^{*})\right)^{1/\lambda}
		&\le
		\sum_{s=1}^d \theta_s^{1/\lambda}(z_1^*,\ldots,z_s^*,z_{s+1}^0,\ldots,z_d^0)
		\le
		\frac2n \sum_{s=1}^d \sum_{s \in \setu \subseteq \{1\mathpunct{:}d\}} 
		\gamma_\setu^{1/\lambda}
		\mu_{\alpha,\lambda}^{|\setu|}
		\\
		&=
		\frac2n \sum_{s=1}^d 
		\left(
		\sum_{\setu \subseteq \{1\mathpunct{:}d\} \setminus \{s\}} 
		\gamma_\setu^{1/\lambda}
		\mu_{\alpha,\lambda}^{|\setu|}
		\right)
		\left(
		\gamma_s^{1/\lambda}
		\mu_{\alpha,\lambda}
		\right)
		\\
		&\le
		\frac2n
		\left(
		\sum_{s=1}^d 
		\gamma_s^{1/\lambda}
		\right)
		\mu_{\alpha,\lambda}
		\max_{s=1,\ldots,d} \left(
		\sum_{\setu \subseteq \{1\mathpunct{:}d\} \setminus \{s\}} 
		\gamma_\setu^{1/\lambda}
		\mu_{\alpha,\lambda}^{|\setu|}
		\right)
		\\
		&=
		\frac2n
		\left(
		\sum_{s=1}^d 
		\gamma_s^{1/\lambda}
		\right)
		\mu_{\alpha,\lambda}
		\max_{s=1,\ldots,d} \left(
		\prod_{s \ne j = 1}^d
		\left(
		1+
		\gamma_j^{1/\lambda}
		\mu_{\alpha,\lambda}
		\right)
		\right)
		\\
		&=
		\frac2n
		\left(
		\sum_{s=1}^d 
		\gamma_s^{1/\lambda}
		\right)
		\mu_{\alpha,\lambda}
		\max_{s=1,\ldots,d} \left(
		\frac{
			\prod_{j = 1}^d
			\left(
			1+
			\gamma_j^{1/\lambda}
			\mu_{\alpha,\lambda}
			\right)
		}{
		1+
		\gamma_s^{1/\lambda}
		\mu_{\alpha,\lambda}
	}
	\right)
	\\
	&=
	\frac2n
	\left(
	\sum_{s=1}^d 
	\gamma_s^{1/\lambda}
	\right)
	\prod_{j = 1}^d
	\left(
	1+
	\gamma_j^{1/\lambda}
	\mu_{\alpha,\lambda}
	\right)
	\mu_{\alpha,\lambda}
	\max_{s=1,\ldots,d} \left(
	1+
	\gamma_s^{1/\lambda}
	\mu_{\alpha,\lambda}
	\right)^{-1}
	\, .
	\end{align*}
	
	To show that the summability condition $\sum_{j=1}^\infty \gamma_j^{1/\lambda} < \infty$ gives a bound independent of $d$ we note that
	\begin{align*}
	\prod_{j = 1}^d \left(1 + \gamma_j^{1/\lambda} \mu_{\alpha,\lambda} \right)
	< 
	\infty
	\quad
	\text{if and only if}
	\quad
	\log\left( \prod_{j = 1}^d \left(1 + \gamma_j^{1/\lambda} \mu_{\alpha,\lambda} \right) \right)
	<
	\infty
	\,.
	\end{align*}
	Now using that $\log(1 + x) \le x$ for $x > -1$, we find that
	\begin{align*}
	\log\left( \prod_{j = 1}^d \left(1 + \gamma_j^{1/\lambda} \mu_{\alpha,\lambda} \right) \right)
	=
	\sum_{j=1}^d \log \left(1 + \gamma_j^{1/\lambda} \mu_{\alpha,\lambda} \right)
	\le
	\mu_{\alpha,\lambda} \sum_{j=1}^d \gamma_j^{1/\lambda}
	\le
	\mu_{\alpha,\lambda} \sum_{j=1}^\infty \gamma_j^{1/\lambda}
	,
	\end{align*}
	which implies the result.
\end{proof}

\subsection{The weighted Sobolev space}		

Again, let $\bsgamma=\{\gamma_j\}$ and $\bsbeta = \{\beta_j\}$ be two weight sequences.
There is a close relationship between the weighted Korobov space with smoothness parameter $\alpha = 1$ and the shift-averaged weighted Sobolev space. The shift-invariant kernel of the weighted Sobolev space with anchor $\bsa = (a_1,\ldots,a_d)$ of $d$-variate functions is given by
\begin{align*}
K_{d,\bsgamma,\bsbeta}^{*}(\bsx,\bsy) 
&=  
\prod_{j=1}^d \left( \hat{\beta}_j + \hat{\gamma}_j \sum_{0 \ne h=-\infty}^{\infty}
\frac{\mathrm{e}^{2\pi \mathrm{i} \, h (x_j - y_j)}}{|h|^2} \right)\, ,
\end{align*}
where $\hat{\beta}_j = \beta_j + \gamma_j \left(a_j^2 - a_j + \frac{1}{3}\right)$ and $\hat{\gamma}_j = \frac{\gamma_j}{2\pi^2}$ for anchor values $a_j \, $. Furthermore, for $c_j = a_j^2 - a_j + \frac{1}{3}$ the shift-averaged squared worst-case error $\hat{e}_{n,d}^2(\bsz)$ with generating vector $\bsz$ takes the following form	
\begin{align*}
\hat{e}_{n,d}^2(\bsz) 
&=
-\prod_{j=1}^{d} \left(\beta_j + \gamma_j \, c_j \right) + \frac{1}{n} \sum_{k=0}^{n-1} \prod_{j=1}^d 
\left( \beta_j + \gamma_j \left[B_2\left(\left\{\frac{k \, z_j}{n}\right\}\right) + c_j \right] \right) \\ 
&=
- \prod_{j=1}^{d} \hat{\beta}_j + \frac{1}{n} \sum_{k=0}^{n-1} \prod_{j=1}^d \left( \hat{\beta}_j + \hat{\gamma}_j \sum_{0 \ne h=-\infty}^{\infty} \frac{\mathrm{e}^{2\pi \mathrm{i} \, k h z_j/n }}{h^2} \right) \, .
\end{align*}
Additionally, the initial worst-case error in the weighted Sobolev space is given by
\begin{align*}
\hat{e}_{0,d}(0,K_{d,\bsgamma,\bsbeta}) 
&= 
\prod_{j=1}^d \hat\beta_j^{1/2}
=
\prod_{j=1}^{d} \left(\beta_j + \gamma_j\left(a_j^2 - a_j + \frac{1}{3}\right)\right)^{1/2} \, .
\end{align*}

Since these are precisely the worst-case error expressions as for the weighted Korobov space with $\alpha = 1$ and weights $\hat{\beta_j}$ 
and $\hat{\gamma_j}$, we obtain similar error bounds as before. \\

\begin{theorem}  \label{thm:ErrBouSob}
	Let $n$ be a prime number and $\bsz^0 = (z_1^0,\ldots,z_d^0) \in \bbZ_n^d$ be an arbitrary initial vector.
	Furthermore, denote by $\bsz^{*} = (z_1^{*},\ldots,z_d^{*})  \in \bbZ_n^d$ the generating vector constructed by the successive coordinate search method with initial vector $\bsz^0$. Then the squared worst-case error $\hat{e}^2_{n,d}(\bsz^{*})$ in the shift-averaged (anchored) Sobolev space with kernel $K^*_{d,\bsgamma,\bsbeta}$ satisfies
	\begin{align*}
	\hat{e}_{n,d}^2(\bsz^{*})
	&\leq 
	\hat{C}_{d,\lambda} \, n^{- \lambda} \quad \text{for all} \quad 1 \leq \lambda < 2
	\, ,
	\end{align*}
	where the constant $\hat{C}_{d,\lambda}$ is given by the expression for $C_{d,\lambda}$ from Theorem~\ref{thm:ErrBouKor} with $\alpha = 1$ and weights $\hat{\beta_j}=\beta_j + \gamma_j \left(a_j^2 - a_j + \frac{1}{3}\right)$ and $\hat{\gamma}_j = \frac{\gamma_j}{2\pi^2}$.
	
	Additionally, if the weights satisfies the summability condition
	\begin{align*}
	\sum_{j=1}^{\infty} \frac{\hat\gamma_j^{\,1/\lambda}}{\hat\beta_j^{\,1/\lambda}} < \infty
	\end{align*}
	then $\hat{C}_{d,\lambda} \le \hat{C}_\lambda < \infty$, and the constant $C_\lambda$ is bounded independent of the dimension~$d$. 
	Hence, the worst-case error $\hat{e}_{n,d}(\bsz^{*})$ can be taken arbitrarily close to $O(n^{-1})$, with the implied constant independent of $n$, and independent of $d$ if the summability condition holds.
\end{theorem}

\begin{proof}
	The theorem follows directly from the previous result in Theorem \ref{thm:ErrBouKor}.
\end{proof}


\section{Numerical results and experiments}\label{sec:numerics}

The idea regarding the SCS algorithm is to obtain generating vectors with smaller error values than obtained by the CBC algorithm, provided we choose a suitable initial vector $\bsz^0 \in \bbZ_n^d$. The formulation of the algorithm suggests that the performance of the SCS construction strongly depends on the starting vector $\bsz^0$ which we select beforehand. In this section we conduct some numerical experiments in the same setting as for the CBC algorithm in order to assess the performance of the SCS algorithm.
For the experiments we prefer to have $n$ distinct points in each dimension and so restrict our generating vector choices to exclude the choice $z_s=0$ for the components of $\bsz$ for prime $n$, i.e., $\bsz \in \{1,\ldots,n-1\}^d$. Allowing the choice $z_s=0$ has effect on the results which depend on the weights since the CBC algorithm can now pick a zero component when the weights $\gamma_j$ decay too slow.

\subsection{Construction methods}

As we do not know how to best choose the initial vectors for the SCS algorithm, we propose 
to start from randomly selected initial vectors. This is different from the randomized  CBC construction, see, e.g., 
\cite{L'Ecuyer_theory:2009}, where in each minimization step the number of possible candidates $z_s$ is restricted to $r$ random integers in $\{1,\ldots,n-1\}$. We consider the following two methods.\\ \\
\textbf{1. Uniform random vectors + SCS algorithm:}
Choose $q$ initial vectors $\bsz^0 \in \bbZ_n^d$ at random, apply the fast SCS algorithm to them and then select the one with the smallest worst-case error $e_{n,d}(\bsz)$. \\ \\
\textbf{2. Korobov-type generating vector + SCS algorithm:}
Take $q$ randomly chosen Korobov-type generating vectors $\bsz^0 = \bsz(a) \equiv (a^0,a^1,\ldots,a^{d-1}) \pmod{n}$, with $a \in \{1,\ldots,n-1\}$, as initial vectors, apply the fast SCS algorithm to them and then select the one with the smallest worst-case error $e_{n,d}(\bsz)$. \\ \\
As the successive coordinate search algorithm has time complexity $O(d \, n \, \log(n))$, both proposed construction methods
have time complexity $O(q \, d \, n \, \log(n))$. 

\begin{remark}
	The obvious candidate for the initial vector $\bsz^0$ would of course be the generating vector constructed by the CBC method
	since by Proposition \ref{prop:ImproveSCS} one would construct $\bsz^1$ such that $e_{n,d}(\bsz^1) \leq e_{n,d}(\bsz^0)$.
	However, experiments show that in most cases the CBC vector is a fixed point with respect to the SCS method, i.e., applying 
	the SCS algorithm to the CBC vector $\bsz^0$ leaves the coordinates of $\bsz^0$ unchanged.  
	Thus, this approach yields usually no further improvement.
\end{remark}

\subsection{Exhaustive search in low dimensions}

In order to test the effectivity of our method we perform some numerical experiments in low dimensions and for a low number of points. 
Here, we can compute the best generating vector for the respective function space via an exhaustive search over the
full set $\bbZ_n^d$ and then compare its worst-case error to the error values of the generating vectors 
obtained by our method. \\ \\
For the weighted unanchored Sobolev space the squared worst-case error is given by
\begin{align*}
e_{n,d}^2(\bsz) 
&=
-\prod_{j=1}^{d} \beta_j + \frac{1}{n} \sum_{k=0}^{n-1} \prod_{j=1}^{d} \left( \beta_j + \gamma_j \, B_2\left( \left\{ 
\frac{k \, z_j}{n} \right\} \right) \right) \, ,
\end{align*}    
where $B_2(x) = x^2 - x +\frac{1}{6}$ denotes the Bernoulli polynomial of degree two. Furthermore, $\bsz_{\text{full}}$ denotes the
generating vector obtained by the full exhaustive search, $\bsz_{\text{cbc}}$ denotes the generating vector obtained via the 
component-by-component construction and $\bsz^*_{\text{rand}}$ and $\bsz^*_{\text{kor}}$ are the best generating vectors obtained out of $q=100$ initial random choices by the above construction methods~1 and~2, respectively.
For two different weight sequences $\gamma_j$ and a selection of prime $n$ we obtain the results in Tables~\ref{tbl:results1} and~\ref{tbl:results2}, where $\gamma_j = (0.95)^j$ and $\gamma_j = (0.7)^j$, respectively.
To be able to find the global minimum $\bsz_{\text{full}}$ over the whole set we limited the dimensionality to $d=5$ and the number of points to $n \le 199$.
This leads to exhaustive searches over about $6$ to $96$~million possible choices for $\bsz$, where we used the symmetry of the kernel and the fact that we only need to consider generating vectors with $z_1 = 1$ since multiplication by the multiplicative inverse of the first component normalizes any generating vector to have $z_1=1$ and this is just a reordering of the cubature nodes.

{\centering
\begin{minipage}{\linewidth}
\begin{table}[H]
	\caption{Weighted unanchored Sobolev space with $d=5, \beta_j=1, \gamma_j=(0.95)^j, q=100$}\label{tbl:results1}
	\centering
	\begin{tabular}{p{2cm}p{2cm}p{2.5cm}p{2cm}p{2cm}}
		\hline\noalign{\smallskip}	
		{$n$} & 
		$e_{n,d}(\bsz^*_{\text{kor}})$ & 
		{$e_{n,d}(\bsz^*_{\text{rand}})$} & 
		{$e_{n,d}(\bsz_{\text{cbc}})$} & 
		{$e_{n,d}(\bsz_{\text{full}})$} \\ \midrule
		101 & 2.6003e-02 & 2.6000e-02 & 2.6022e-02 & 2.6000e-02 \\
		127 & 2.1794e-02 & 2.1834e-02 & 2.2180e-02  & 2.1751e-02 \\
		139 & 2.0016e-02 & 2.0010e-02 & 2.0493e-02  & 1.9999e-02 \\ \midrule
		151 & 1.8886e-02 & 1.8893e-02 & 1.9175e-02  & 1.8843e-02 \\
		181 & 1.5963e-02 & 1.5937e-02 & 1.6453e-02  & 1.5928e-02 \\
		199 & 1.4813e-02 & 1.4808e-02 & 1.5368e-02  & 1.4802e-02 \\ \bottomrule
	\end{tabular}
\end{table}
\end{minipage}
\par
}

{\centering
\begin{minipage}{\linewidth}
\begin{table}[H]
	\caption{Weighted unanchored Sobolev space with $d=5, \beta_j=1, \gamma_j=(0.7)^j, q=100$}\label{tbl:results2}
	\centering
	\begin{tabular}{p{2cm}p{2cm}p{2.5cm}p{2cm}p{2cm}}
		\hline\noalign{\smallskip}	
		{$n$} & 
		$e_{n,d}(\bsz^*_{\text{kor}})$ &
		$e_{n,d}(\bsz^*_{\text{rand}})$ &
		{$e_{n,d}(\bsz_{\text{cbc}})$} & 
		{$e_{n,d}(\bsz_{\text{full}})$} \\ \midrule
		101 & 1.0721e-02 & 1.0695e-02 & 1.0878e-02 & 1.0695e-02 \\
		127 & 8.7079e-03 & 8.6296e-03 & 8.6700e-03  & 8.6275e-03 \\
		139 & 8.0567e-03 & 8.0439e-03 & 8.0724e-03 & 8.0439e-03 \\ \midrule
		151 & 7.4913e-03  & 7.4913e-03  & 7.5295e-03  & 7.4913e-03 \\
		181 & 6.26793e-03 & 6.2594e-03 & 6.3898e-03  & 6.2421e-03 \\
		199 & 5.7456e-03 & 5.7682e-03 & 5.8758e-03  & 5.7352e-03 \\ \bottomrule
	\end{tabular}
\end{table}
\end{minipage}
\par
}
\vspace{15pt}
The results in Table~1 and~2 show that, even for a moderate value of $q$, the randomized SCS method generates lattice rules which
have a smaller worst-case error than the one obtained via the CBC construction. Additionally, we see that our
method generates worst-case errors that lie in the region of the smallest worst-case error $e_{n,d}(\bsz_{\text{full}})$
and sometimes even constructs the best possible lattice rule.
Although we only show two small tables here, similar results were observed for other test cases as well.
In particular, we considered weight sequences of the form $\gamma_j = q^j$ with $0 < q < 1$ and $\gamma_j = j^{-k}$
with $k \in \{2,3,4\}$, for additional results see \cite{Master_thesis:2015}. The experiments showed that
the SCS algorithm outperforms the CBC construction when the decay of the weight sequence $\gamma_j$ is slow.

\subsection{Numerical experiments for higher dimensions}

In higher dimensions and/or for higher number of points it is not possible to perform an exhaustive search in order to obtain a reference value to measure the quality of the constructed generating vectors. Thus, we compare the outcome of the SCS method with
the generating vector constructed by the CBC algorithm. Additionally, the empirical numerical results suggested that the use of Korobov-type initial vectors is to be preferred over uniform random vectors and we will therefore only consider Korobov-type initial vectors in this section.
We denote by $\overline{e_{n,d}(\bsz_{\text{kor}})}$ the average over the $q$ random choices of the worst-case errors of the SCS constructed vectors $\bsz_{\text{kor}}$ and with $e_{n,d}(\bsz_{\text{kor}}^*)$ the best over the $q$ random choices.

\begin{table}[H]
	\caption{Weighted Korobov space with $d=100, \alpha=1, \beta_j=\frac{2}{3}, \gamma_j=\frac{2}{3} (0.95)^j, q=100$}\label{tbl:k100-1}
	\centering
	\begin{tabular}{p{2cm}p{2cm}p{2.5cm}p{2cm}}
		\hline\noalign{\smallskip}	
		{$n$} & {$\overline{e_{n,d}(\bsz_{\text{kor}})}$} & $e_{n,d}(\bsz_{\text{kor}}^*)$ & {$e_{n,d}(\bsz_{\text{cbc}})$} \\ \midrule
		1009 & 1.6554e-02 & 1.6221e-02 & 1.6566e-02 \\
		2003 & 1.1759e-02 &  1.1474e-02 & 1.1719e-02 \\
		4001 & 8.3025e-03 &  8.1204e-03 &  8.2869e-03 \\ 
		8009 & 5.8655e-03 & 5.7730e-03 & 5.8500e-03 \\
		32003 & 2.9320e-03 & 2.8874e-03 &  2.9301e-03 \\ \bottomrule
	\end{tabular}
\end{table}	

\begin{table}[H]
	\caption{Weighted Korobov space with $d=100, \alpha=1, \beta_j=1, \gamma_j=(0.7)^j, q=100$}\label{tbl:k100-2}
	\centering
	\begin{tabular}{p{2cm}p{2cm}p{2.5cm}p{2cm}}
		\hline\noalign{\smallskip}	
		{$n$} & {$\overline{e_{n,d}(\bsz_{\text{kor}})}$} & $e_{n,d}(\bsz_{\text{kor}}^*)$ & {$e_{n,d}(\bsz_{\text{cbc}})$} \\ \midrule
		1009 & 3.1185e-01 & 3.0834e-01 & 3.0931e-01 \\
		2003 & 2.0902e-01 &  2.0661e-01 & 2.0708e-01 \\
		4001 & 1.3894e-01 &  1.3713e-01 &  1.3658e-01 \\ 
		8009 & 9.1757e-02 & 9.0445e-02 & 8.9611e-02 \\
		32003 & 3.9467e-02 & 3.8763e-02 &  3.8528e-02 \\ \bottomrule
	\end{tabular}
\end{table}	

The numerical results presented in Tables~\ref{tbl:k100-1} and~\ref{tbl:k100-2} are for a Korobov space with $d=100, \alpha=1$ and two different choices of weights, being $\beta_j=\frac{2}{3}$ and $\gamma_j=\frac{2}{3} (0.95)^j$, and $\beta_j=1$ and $\gamma_j=(0.7)^j$, respectively, both with $q=100$ random initial Korobov-type vectors. Our experiments show that the SCS method can construct good lattice rules for high dimensions and large $n$.
For our choice of parameters, the SCS algorithm performs moderately better than the CBC construction when 
the weight sequence $\bsgamma=\{\gamma_j\}_{j=1}^d$ is slowly decaying, as can be seen by comparing
the relative difference between $e_{n,d}(\bsz_{\text{kor}}^*)$ and {$e_{n,d}(\bsz_{\text{cbc}})$} for the two different 
weight sequences in Table~3 and~4. For a more extensive analysis of this behaviour we refer again to \cite{Master_thesis:2015}
where a wider range of weight sequences is considered.

{\centering
\begin{minipage}{\linewidth}
\begin{figure}[H]
	\centering
	\caption{
	Numerical results of the SCS method in the weighted Korobov space with $d=100$, $\alpha=1$, $\beta_j = \frac23$, $\gamma_j = \frac23 (0.95)^j$ where $n=4001$ and $q=300$ in comparison to the CBC algorithm. 
	}
	\label{fig:Scatter}
	\includegraphics[width=0.88\textwidth]{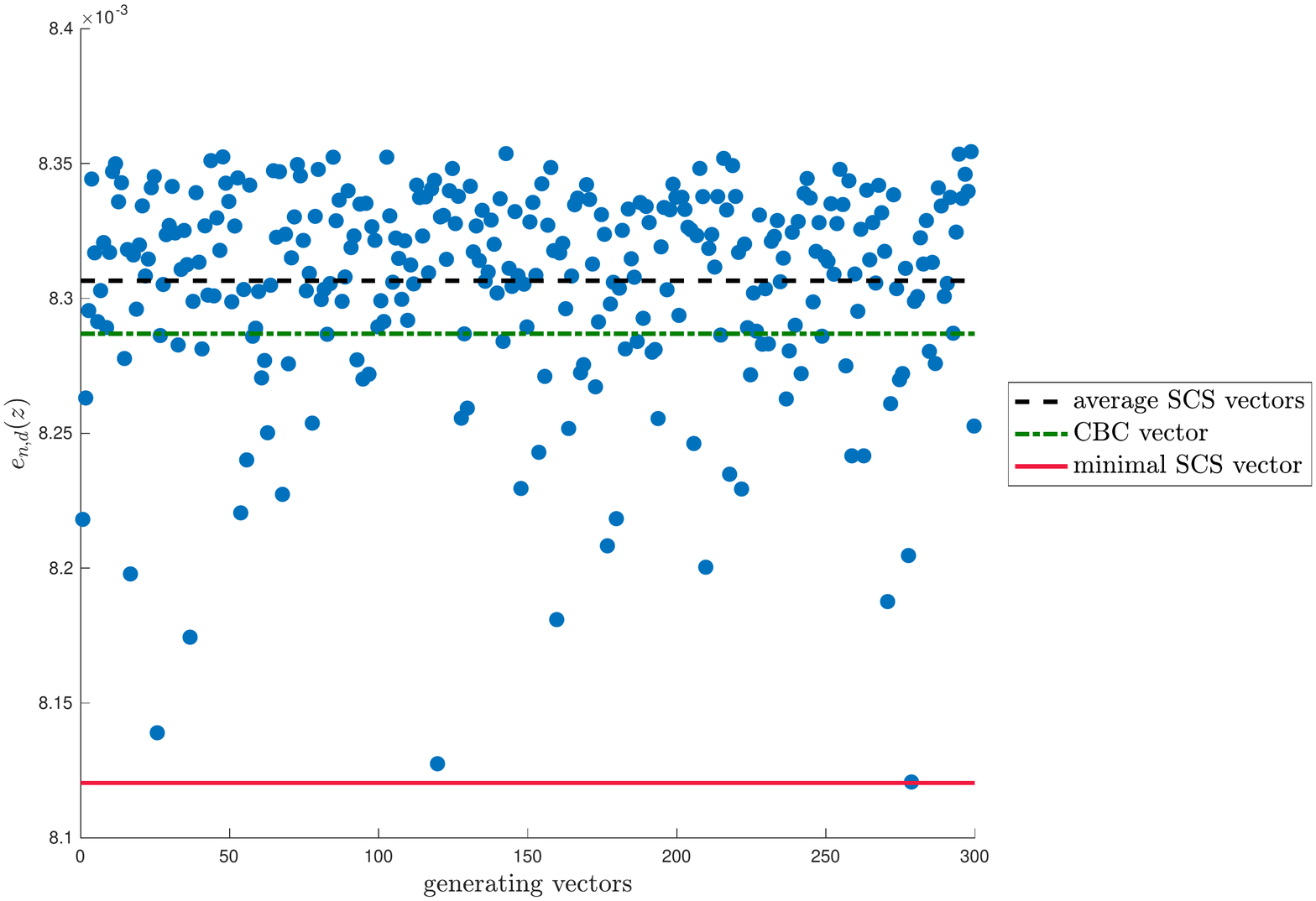}
\end{figure}
	\end{minipage}
	\par
}
Fig. \ref{fig:Scatter} illustrates the performance of the SCS method compared to the CBC method.
The blue dots represent the worst-case error values of lattice rules with $n=4001$ points constructed by the SCS method with $q=300$ Korobov-type initial vectors. The minimal error amongst the constructed lattice rules and the average over the $q$ random seed choices is indicated by the red or black line, respectively. The error corresponding to the CBC method is indicated by the green line. From the figure it becomes evident that the CBC algorithm outperforms the average of the SCS algorithm applied to randomly selected Korobov-type rules, but the best SCS results clearly win over the generating vector constructed by the CBC method. 

\section{Conclusion}

The results and experiments in the previous section, see \cite{Master_thesis:2015} for additional results, showed that it is possible to use the successive coordinate search algorithm to construct good generating vectors for rank-1 lattice rules. 
They also confirmed that randomized methods based on the SCS construction can provide generating vectors with smaller worst-case errors than the CBC vector. However, the computational cost of the SCS method can be several times higher while the gained improvement depends on the weight sequence $\bsgamma$. Future research could help to find a selection criterion for the starting vector $\bsz^0$ in order to reduce the construction cost of the SCS algorithm.
The SCS algorithm should further be regarded as a generalization of the existing component-by-component construction rather than a completely new algorithm. Due to the formulation of the successive coordinate search method it can also be used to improve existing lattice rules. Numerical experiments show that the improvements of the SCS method are higher when the decay of the weights $\bsgamma=\{\gamma_j\}_{j=1}^d$ is slow.

\vspace{15pt}
\noindent
\textbf{Acknowledgements}
We thank Peter Kritzer for some useful comments and discussions about the manuscript and we acknowledge financial support from the KU Leuven research fund (OT:3E130287 and C3:3E150478).

\bibliographystyle{spmpsci}
\bibliography{mybibfile}

\end{document}